%% file: main.tex
\documentclass[12 pt]{article}
\usepackage[margin = 2.4cm]{geometry}
\usepackage{setspace}
\usepackage{amsmath, amsthm, amssymb,enumerate}
\usepackage{tikz}
\usetikzlibrary{decorations.pathreplacing}
\usepackage{fullpage}
\usepackage{enumitem}
\usepackage[bookmarks = false, hidelinks]{hyperref}
\usepackage{xcolor}
\usepackage{thm-restate}
\usepackage{cleveref}
\usepackage{caption}
\title{The rainbow saturation number is linear}
\author{Natalie Behague\thanks{Mathematics Institute, University of Warwick, Coventry, CV4~7AL, UK. This research was supported by a PIMS postdoctoral fellowship. Email: \href{mailto:natalie.behague@warwick.ac.uk}{natalie.behague@warwick.ac.uk}.} \and Tom Johnston\thanks{School of Mathematics, University of Bristol, Bristol, BS8~1UG, UK and Heilbronn Institute for Mathematical Research, Bristol, UK. Email: \href{mailto:tom.johnston@bristol.ac.uk}{tom.johnston@bristol.ac.uk}.} \and Shoham Letzter\thanks{Department of Mathematics, University College London, Gower Street, London WC1E~6BT, UK. Email: \href{mailto:s.letzter@ucl.ac.uk}{s.letzter@ucl.ac.uk}. Research supported by the Royal Society.} \and Natasha Morrison\thanks{Department of Mathematics and Statistics, University of Victoria, Canada. Research supported by NSERC Discovery Grant RGPIN-2021-02511 and NSERC Early Career Supplement DGECR-2021-00047. Email: \href{mailto:nmorrison@uvic.ca}{nmorrison@uvic.ca}.} 
 \and Shannon Ogden
 \thanks{Department of Mathematics and Statistics, University of Victoria, Canada. Supported by NSERC CGS M. Email: \href{mailto:sogden@uvic.ca}{sogden@uvic.ca}.}}

\newtheoremstyle{case}{}{}{\normalfont}{}{\itshape}{:}{ }{}

\newtheorem{thm}{Theorem}[section]
\newtheorem{lem}[thm]{Lemma}
\newtheorem{prop}[thm]{Proposition}

\theoremstyle{definition}

\newtheorem{qn}[thm]{Question}

\newtheorem{claim}[thm]{Claim}
\newtheorem{constr}[thm]{Construction}

\newtheoremstyle{case}{}{}{\normalfont}{}{\itshape}{\normalfont:}{ }{}

\theoremstyle{case}

\Crefname{prop}{Proposition}{Propositions}

\DeclareMathOperator{\sat}{sat}
\DeclareMathOperator{\wsat}{wsat}
\DeclareMathOperator{\rwsat}{rwsat}
\DeclareMathOperator{\rsat}{rsat}
\DeclareMathOperator{\prsat}{prsat}
\DeclareMathOperator{\ex}{ex}

\newcommand{\R}{\mathfrak{R}}
\newcommand{\N}{\mathbb{N}}
\newcommand{\xy}{x|y}

\numberwithin{equation}{section}

\begin{document}

\maketitle

\begin{abstract}
    Given a graph $H$, we say that an edge-coloured graph $G$ is $H$-rainbow saturated if it does not contain a rainbow copy of $H$, but the addition of any non-edge in any colour creates a rainbow copy of $H$. The rainbow saturation number $\rsat(n,H)$ is the minimum number of edges among all $H$-rainbow saturated edge-coloured graphs on $n$ vertices. We prove that for any non-empty graph $H$, the rainbow saturation number is linear in $n$, thus proving a conjecture of Gir\~{a}o, Lewis, and Popielarz. In addition, we also give an improved upper bound on the rainbow saturation number of the complete graph, disproving a second conjecture of Gir\~{a}o, Lewis, and Popielarz.
\end{abstract}

\section{Introduction} 
\label{Sec:Intro}

For a fixed graph $H$, we say that a graph $G$ is \emph{$H$-saturated} if it does not contain $H$ as a subgraph, but adding any extra edge creates a copy of $H$ as a subgraph. 
The \emph{saturation number} of $H$, denoted by $\sat(n,H)$, is the minimum number of edges in an $H$-saturated graph $G$ on $n$ vertices.
This can be thought of as a dual to the classical Tur\'{a}n extremal number $\ex(n,H)$ that counts the maximum number of edges in an $H$-free graph on $n$ vertices. Since a maximal $H$-free graph is $H$-saturated, $\ex(n,H)$ equivalently counts the maximum number of edges in an $H$-saturated graph on $n$ vertices. The saturation number in contrast counts the minimum number of edges among such graphs.

The saturation number was first studied independently by Zykov \cite{Zykov} and  Erd\H{o}s, Hajnal, and Moon \cite{EHM}, who proved that $\sat(n,K_r) = (r-2)(n-1) - \binom{r-2}{2}$, where as usual $K_r$ denotes the complete $r$-vertex graph. Later, 
K\'{a}szonyi and Tuza \cite{KT} proved that the saturation number of every graph $H$ is linear in $n$, that is, $\sat(n,H) = O_H(n)$. For more information and many other results related to the saturation number see the survey of Faudree, Faudree, and Schmitt \cite{FFS}. In this paper we provide analogous results to these for the rainbow saturation number: a variant of the saturation number for edge-coloured graphs. 

The generalisation of saturation to edge-coloured graphs was first considered by Hanson and Toft \cite{HT}. Following this, Barrus, Ferrara, Vandenbussche, and Wenger \cite{BFVW} considered the particular case of $t$-rainbow saturation, where there are exactly $t$ colours available. Gir\~{a}o,  Lewis, and Popielarz \cite{GLP} initiated the study of the natural generalisation to the case where the palette of colours available is unlimited, rather than bounded by $t$. This is the focus of our paper.

A $t$\emph{-edge-coloured} graph is an ordered pair $(G,c)$, where $c$ is a (not necessarily proper) edge-colouring of the graph $G$ using colours from $[t]=\{1,2,...,t\}$. An edge-colouring of a graph is said to be \emph{rainbow} if every edge is assigned a distinct colour. We say that a $t$-edge-coloured graph $(G,c)$ is $(H,t)$\emph{-rainbow saturated} if $(G,c)$ does not contain a rainbow copy of $H$, but the addition of any non-edge in any colour from $[t]$ creates a rainbow copy of $H$ in $G$. Note that this requires $t\geq |E(H)|$. The $t$\emph{-rainbow saturation number} of $H$, denoted by $\rsat_t(n,H)$, is the minimum number of edges in an $(H,t)$-rainbow-saturated graph on $n$ vertices. 

When the number of possible edge colours is infinite (say the set of colours is $\N$), an edge-coloured graph $(G,c)$ is $H$\emph{-rainbow saturated} if $(G,c)$ does not contain a rainbow copy of $H$, but the addition of any non-edge in any colour from $\N$ creates a rainbow copy of $H$. The \emph{rainbow saturation number} of $H$, denoted by $\rsat(n,H)$, is the minimum number of edges in an $H$-rainbow saturated graph on $n$ vertices. 

In \cite{GLP}, Gir\~{a}o, Lewis, and Popielarz conjectured that, like ordinary saturation numbers, the rainbow saturation number of any non-empty graph $H$ is at most linear in $n$, and they proved this for graphs with some particular properties. 

\begin{thm}[Gir\~ao, Lewis, and Popielarz; Theorem 2~(4) in \cite{GLP}]
	\label{GLP construction}Let $H$ be a graph with a non-pendant edge that is not contained in a triangle. Then $\rsat(n,H) = O(n)$.
\end{thm}
\noindent In fact, \Cref{GLP construction} is a consequence of a stronger result that proves $\rsat_t(n,H) = O(n)$ for any such $H$ and any $t \ge e(H)$.

Our main result proves that the rainbow saturation number of every graph is linear, thus confirming the conjecture of Gir\~{a}o, Lewis, and Popielarz~\cite{GLP}.

\begin{restatable}{thm}{thmLinear}
\label{linear rainbow}
	Every non-empty graph $H$ satisfies $\rsat(n,H)=O(n)$.
\end{restatable} 

Interestingly, when $t$ is fixed the $t$-rainbow saturation number is not in general linear in $n$.  Barrus, Ferrara, Vandenbussche, and Wenger \cite{BFVW} proved in particular that for all $t \ge |E(H)|$, $\rsat_t(n,H) \ge c_1\frac{n\log{n}}{\log{\log{n}}}$ whenever $H$ is a complete graph, and $\rsat_t(n,H) \ge c_2n^2$ if $H$ is a star, where $c_1$ and $c_2$ are constants depending only on $t$ and the size of $H$. 

The lower bound on the $t$-rainbow saturation number for complete graphs was improved independently by Ferrara, Johnston, Loeb, Pfender, Schulte, Smith, Sullivan, Tait, and Tompkins \cite{FJLPSSSTT},  by Kor\'{a}ndi \cite{Korandi}, and by Gir\~{a}o,  Lewis, and Popielarz \cite{GLP} to  $\rsat_t(n,K_r) \ge c_1n\log{n}$, which is tight up to a constant factor.

Gir\~{a}o, Lewis, and Popielarz also conjectured that for any complete graph $K_r$ where $r\geq 3$, the rainbow saturation number satisfies $\rsat(n, K_r) = 2(r-2)n + O(1)$ when $n\geq 2(r-2)$. We prove that $\rsat(n, K_r)$ is in fact significantly less than this conjectured value.

\begin{restatable}{thm}{thmClique}
\label{clique}
Let $r \ge 3$. Then $\rsat(n, K_r) \le (r + 2\sqrt{2r})n + c_r$, where $c_r$ is a constant depending only on $r$.
\end{restatable}

In Section \ref{Sec:Prelim}, we establish some useful basic results regarding rainbow saturation, including \Cref{disconnect,pendant} which allow us to assume that any counterexample to \Cref{linear rainbow} is connected and has no pendant edges.
We describe the construction used to prove \Cref{linear rainbow}  in Section \ref{Sec:Constr}: given a graph $H$, we construct an edge-coloured graph on $n$ vertices with $O(n)$ edges such that the addition of (almost) any non-edge creates a rainbow copy of $H$. 
In Section \ref{Sec:no rainbow}, we show that if $e$ is an edge contained in a triangle then the graph produced by this construction does not contain a rainbow copy of $H$. Together with \Cref{GLP construction}, this completes the proof of \Cref{linear rainbow}.
Finally, we prove \Cref{clique} in Section \ref{sec:Cliques}.

\section{Preliminaries} 
\label{Sec:Prelim}
In this section we provide preliminary results that deal with some easy classes of graphs. We begin by establishing that the existence of an edge-coloured graph $G$ on $n$ vertices with $O(n)$ edges that is ``almost'' $H$-rainbow-saturated (save for at most linearly many problematic non-edges) is enough to prove that $\rsat(n,H)=O(n)$. For a colour $c$, we say that a non-edge $e \in E(\overline{G})$ is \emph{$c$-bad} for $G$ if the addition of $e$ in colour $c$ to $G$ does not create a rainbow copy of $H$. We say that a non-edge is \emph{bad} if there exists a colour $c$ such that $e$ is $c$-bad for $G$.

\begin{prop}
\label{absorb}
Let $H$ be a graph, and let $G$ be an $n$-vertex edge-coloured graph. Suppose that $G$ has $m$ bad non-edges and does not contain a rainbow copy of $H$. Then $\rsat(n,H)\le |E(G)| + m$. 
\end{prop}

\begin{proof} 
	Let $\{e_1,\ldots, e_m\}$ be the set of bad non-edges of $G$. We construct an $n$-vertex graph $G_m$ and colouring $\chi$ of $E(G_m)$ as follows: set $G_0:= G$, where each edge of $G_0$ is coloured as it is in $G$. Consider each $i\in [m]$ in turn. If there is a colour $c$ such that $e_i$ is $c$-bad for $G_{i-1}$, then we define $G_i$ to be $G_{i-1} \cup \{e_i\}$, set $\chi(e_i) = c$ (that is, $e_i$ receives colour $c$), and keep the colours of the other edges as in $G_{i-1}$. Note that there may be multiple such $c$, in which case we pick one arbitrarily. Otherwise, if $e_i$ is not bad for $G_{i-1}$, set $G_i := G_{i-1}$. Observe that $G_m$ is $H$-rainbow-saturated, and since we added at most $m$ edges during this process, $G_m$ has at most $|E(G)|+m$ edges. Hence, $\rsat(n,H)\le |E(G)| + m$ as required. 
\end{proof}

A direct result of \Cref{absorb} shows that in order to prove \Cref{linear rainbow} it suffices to consider connected graphs $H$. 

\begin{prop}
\label{disconnect}
	Let $H$ be a disconnected graph and let $H'$ be a component of $H$ that has the most vertices, and subject to this, the most edges. If $\rsat(n,H')=O(n)$, then $\rsat(n,H)=O(n)$. 
\end{prop}

\begin{proof}
Let $s$ be the number of components of $H$ which are isomorphic to $H'$. As $\rsat(n,H')=O(n)$, we have $\rsat(n-m,H')=O(n)$, where $m$ is given by $m:=2 \left(|V(H)| - |V(H')|\right) - (s-1)$. Let $(G,\chi)$ be an $H'$-rainbow-saturated graph on $n-m$ vertices with as few edges as possible. 

Define a graph $G^*$ on $n$ vertices as follows: let $H_2$ be obtained from $H'$ by gluing two disjoint copies of $H'$ together at an arbitrary vertex. Let $G^*$ be the disjoint union of $G$ with $s-1$ disjoint copies of $H_2$ and two disjoint copies of each component of $H$ that is not isomorphic to $H'$. See Figure~\ref{fig:disc} for an example. Let $\chi^*$ be an edge-colouring of $E(G^*)$ where $\chi^*(e) = \chi(e)$ for every edge $e \in G$, and $G^*\setminus V(G)$ is rainbow with colours not used in $\chi$. 

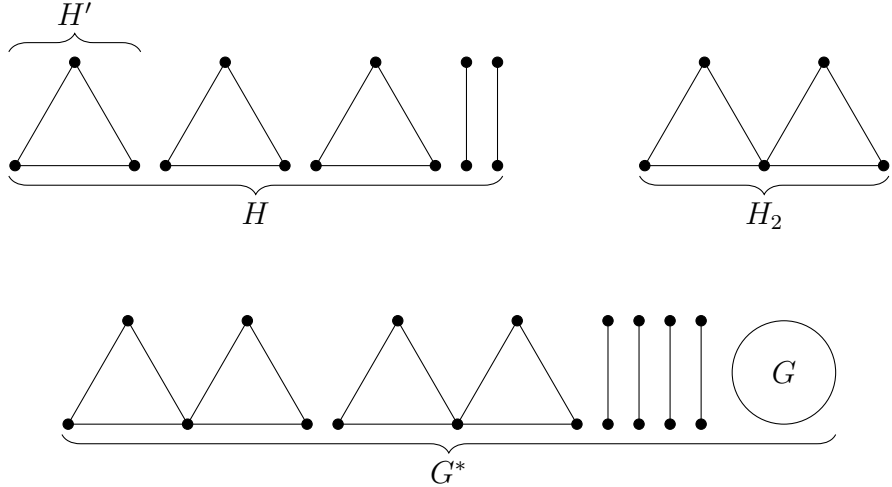
\begin{figure}[htbp]
\centering
\input{fig1}
\caption{An example of the graph $G^*$ created via the construction described in \Cref{disconnect}. Here $s=3$.}\label{fig:disc}
\end{figure}

Observe that $G^*$ contains no rainbow copy of $H$ (by our choice of $H'$, there are at most  $s-1$ disjoint copies of $H'$ in $G^*\setminus G$ and $G$ contains no rainbow copy of $H'$). Moreover, the addition of any non-edge (in any colour) within $G$ creates a rainbow copy of $H$ in $G^*$, since $G^* \setminus V(G)$ contains a rainbow copy of $H \setminus V(H')$ avoiding any given colour. As there are at most $m(n-m)+ \binom{m}{2}=O(n)$ other non-edges in $G^*$, by \Cref{absorb} we have $\rsat(n,H)=O(n)$, as required.
\end{proof}

Therefore, in what follows we may (and will) assume that $H$ is a connected graph. Our next result shows that, if $H$ contains a vertex of degree $1$, then $\rsat(n,H)=O(n)$. 

\begin{prop}
\label{pendant}
Let $H$ be a connected graph with $\delta(H)=1$. Then $\rsat(n,H)=O(n)$.
\end{prop}

\begin{proof}  
	Write $k = |V(H)|$ and $n=q(k-1)+r$, where $0\leq r< k-1$. Define an edge-coloured graph $G$ on $n$ vertices as follows: take $q$ disjoint copies of $K_{k-1}$ and let $\chi$ be a rainbow colouring of them. Finally, add $r$ isolated vertices.
	Note that $G$ has $n$ vertices and $q\binom{k-1}{2} = (n-r)\left(\frac{k-2}{2}\right) = O(n)$ edges. Clearly $G$ does not contain a rainbow copy of $H$. Moreover, since $\delta(H)=1$, if any non-edge in any colour is added between two distinct copies of $K_{k-1}$ in $G$, then a rainbow copy of $H$ is created. Therefore, since $G$ has at most $q(k-1)r + \binom{r}{2} =(n-r)r + \binom{r}{2} =O(n)$ other non-edges, \Cref{absorb} implies that $\rsat(n,H)=O(n)$.
\end{proof}

\section{The construction} 
\label{Sec:Constr}

In this section, we give the construction that lies at the heart of our proof of \Cref{linear rainbow}. Given a graph $H$ and an edge $e\in E(H)$, we create an edge-coloured graph on $n$ vertices with $O(n)$ edges such that the addition of (almost) any edge creates a rainbow copy of $H$. The main work in the proof of \Cref{linear rainbow} will be to show that this graph does not contain a rainbow copy of $H$, which will be done in Section~\ref{Sec:no rainbow}.

\begin{constr}
\label{construction}

Let $H$ be a connected graph on at least three vertices, and let $xy\in E(H)$. 
Write $S := N_H(x) \cap N_H(y)$ and let $T$ be the set of edges in $H$ with one endpoint in $S$ and the other in $\{x,y\}$. 
Define $H'$ to be the graph obtained from $H$ by contracting the edge $xy$ and replacing any multi-edges by single edges. We label the vertices of $H'$ as in $H$, with the single vertex obtained from contracting $xy$ labelled by $\xy$. Let $T'$ be the set of edges in $H'$ between $\xy$ and $S$, and write $H'' := H' \setminus T'$.

For an integer $m \geq 2$, define $F = F(m)$ to be the graph obtained from $H'$ by replacing the vertex $\xy$ with $m$ duplicates of itself, denoted $v_1, v_2,\ldots, v_m$. Write $M := \{v_1, v_2, \ldots, v_m\}$ and label the vertices in $M$ by $\xy$. Label the remaining vertices as in $H$. This means $F$ has $m$ vertices labelled $\xy$ and one vertex labelled $v$ for every $v \in V(H) \setminus \{x, y\}$.

Given any integers $m \geq 2$ and $r \geq \max\{|E(H'')| + 1,2\}$, define the graph $G := G^{H}_{xy}(m,r)$ as follows. Start with $|E(H'')|$ disjoint copies of $F(m)$, indexed by the edges in $H''$ as $F_{e_1}, \dots, F_{e_{|E(H'')|}}$. To these we add $r - |E(H'')|$ copies of $F(m)$, which we index as $F_1,\allowbreak \dots,\allowbreak F_{r - |E(H'')|}$. Finally, we obtain $G$ by identifying the vertices corresponding to $v_i$ for each $i$ in turn. That is, for each $i = 1, \dots, m$, we replace the $r$ vertices $u_1, \dots, u_r$ corresponding to $v_i$ with one vertex and connect this vertex to the vertices in $\bigcup_{j=1}^r N(u_j)$.  Note that the graph $G$ has $r$ copies of each vertex in $V(H') \setminus \{\xy\}$ and $m$ copies of the vertex $\xy$, and that the labelling of the vertices naturally defines a labelling of the edges, where an edge in $G$ whose vertices are labelled $u$ and $v$ is labelled $uv$. By construction, $uv$ is an edge in $H'$.

We now define a colouring $\chi$ of $G$, which we will later show does not contain a rainbow copy of $H$. Let $\chi_{0}$ be a rainbow colouring of $H''$ and fix a colour that is not in $\chi_0$, say black. We extend $\chi_{0}$ to a colouring of $H'$ by colouring the edges in $T'$ with a colour $\star$. Let an edge $uv$ in $G$ be coloured by $\chi_0(e)$, where $e$ is the colour of the edge in $H'$ matching the label of $uv$. For each edge $e$ in $H''$, recolour the edge corresponding to $e$ in $F_e$ with the colour black, and finally recolour each edge coloured with $\star$ with unique colours that are not used elsewhere in $G$. 

\begin{figure}[htbp]
\centering
\input{fig2}
\caption{An example of the construction of the graph $G_{xy}^{H}(m,r)$. The dashed edge represents edges which are originally coloured $\star$ and will be coloured with unique colours by $\chi$. The colour black can be seen replacing other colours according to the index of the copy of $F$.}\label{fig:con}
\end{figure}
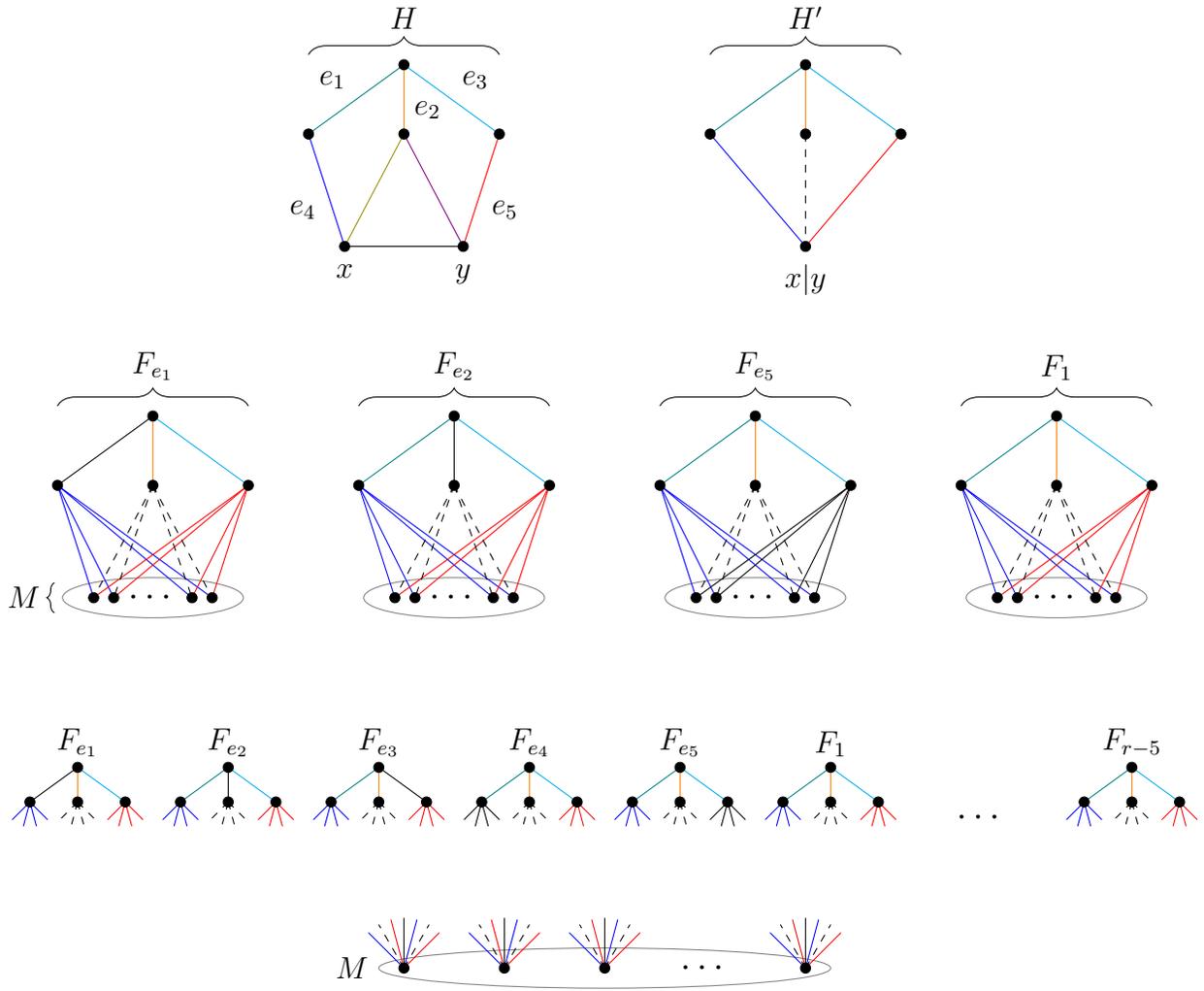

\end{constr}

\Cref{construction} was inspired by the construction used in \cite{GLP} to prove \Cref{GLP construction}. In fact, when the edge $xy$ is not contained in a triangle, $N_H(x)\cap N_H(y)=\emptyset$, and the graph described in \Cref{construction} is equivalent to the construction in \cite{GLP}.

See Figure~\ref{fig:con} for an example of how \Cref{construction} works. In what follows, we will use $G_{xy}^H(n)$ to refer to the edge-coloured graph $(G_{xy}^H(m, r),\chi)$, where $m = n - r(|V(H)| - 2)$ and $r = \max\{|E(H'')| + 1,2\}$; so $G_{xy}^H(n)$ has $n$ vertices. Observe that the addition of any non-edge in any colour within $M$ creates a rainbow copy of $H$. Since $r = O(1)$ and $m = n - O(1)$, there are at most $O(n)$ bad non-edges in $G_{xy}^H(n)$. Therefore, if $G_{xy}^H(n)$ contains no rainbow copy of $H$, then  \Cref{absorb} will yield $\rsat(n,H) = O(n)$, as required.

\section{Completing the proof of Theorem \ref{linear rainbow}} 
\label{Sec:no rainbow}
The goal of this section is to show that $G_{xy}^H$ contains no rainbow copy of $H$. 

\begin{prop}
\label{no rainbow H generalised}
    Let $H$ be a non-empty connected graph. For any edge $xy\in E(H)$ that is contained in a triangle and all integers $m$ and $r$, the graph $G_{xy}^H(m,r)$ does not contain a rainbow copy of $H$.
\end{prop}

Before proving this, we show why it implies \Cref{linear rainbow} (restated here for convenience).

\thmLinear*

\begin{proof}[Proof of \Cref{linear rainbow}]  
	First, note that by \Cref{disconnect}, it suffices to prove that every connected graph $H$ satisfies $\rsat(n, H) = O(n)$. 
	Let $H$ be a non-empty connected graph. If $H$ has a triangle, let $xy$ be an edge contained in a triangle, and write $G := G_{xy}^H(n)$. Then $|E(G)| = O(n)$, and $G$ has no rainbow copies of $H$ by \Cref{no rainbow H generalised}. Moreover, as described at the end of Section \ref{Sec:Constr}, for all but at most $O(n)$ non-edges $e$, the addition of $e$ in any colour creates a rainbow copy of $H$. In particular, $\rsat(n, H) = O(n)$, as required.
	Otherwise, if $H$ is triangle-free, then $\rsat(n, H) = O(n)$, using either \Cref{pendant} if $H$ has a pendant edge or Theorem~\ref{GLP construction} if $H$ has a non-pendant edge.
\end{proof}

\begin{proof}[Proof of \Cref{no rainbow H generalised}]

	Suppose for contradiction that there exists a connected graph $H$ and an edge $xy\in E(H)$ contained in a triangle such that $G_{xy}^H(m,r)$ contains a rainbow copy $H_\R$ of $H$, for some integers $m, r$ with $m \geq 2$ and $r \geq \max\{|E(H'')| + 1,2\}$.
	Take $H$ to have the minimal number of vertices with respect to these properties, and fix appropriate $m$ and $r$ such that $G := G_{xy}^H(m, r)$ has a rainbow copy of $H$.
	Recall that the vertices of $G$ are labelled by vertices in $H'$ (the graph obtained from $H$ by contracting $xy$), $M$ is the set of vertices labelled $\xy$, and that the edges of $G$ are labelled by edges in $H'$.

Since $H_\R$ is connected it must contain at least one vertex in $M$, else it would be contained entirely within a copy of $H\setminus\{x,y\}$, which is impossible as $|V(H_\R)| = |V(H)|$. 

\begin{claim} \label{claim: use every vertex}
    For every $v \in V(H)\setminus\{x,y\}$, the rainbow copy $H_\R$ uses at least one vertex labelled $v$.
\end{claim}
\begin{proof}
    Suppose for contradiction that $v$ is a vertex in $V(H)\setminus\{x,y\}$ such that $H_\R$ uses no vertex labelled $v$. 

	Suppose first that $v$ is the unique common neighbour of $x$ and $y$ in $H$. Then, since $H_\R$ uses no vertex labelled $v$, it uses at most $|E(H)| - 2$ colours: at most $|E(H'')| = |E(H)| - |T| - 1$ colours used by $\chi_0$, and possibly also black. This is a contradiction, as a rainbow copy of $H$ requires $|E(H)|$ colours.
    
	Next, suppose that the graph $H^v := H\setminus \{v\}$ is connected. By the previous paragraph, we may (and will) assume that $v$ is not the unique common neighbour of $x$ and $y$. Consider the edge-coloured graph $G^*$ obtained from $G = G_{xy}^H(m,r)$ by the deletion of every vertex labelled $v$. As $H_\R$ contains no vertex labelled $v$, it is present in $G^*$, and, in particular, $G^*$ contains a rainbow copy of $H^v$. Observe that $G^* = (G_{xy}^{H^v}(m,r), \chi_0')$, where $\chi_0'$ is the restriction of $\chi_0$ to $(H^v)''$ and $r \ge \max\{|E(H^v)| +1, 2\}$. Since $v$ is not the unique common neighbour of $x$ and $y$, the edge $xy$ is in a triangle in $H^v$. Therefore, since $H^v$ is connected, by minimality of $H$ we know that $G_{xy}^{H^v}(m,r)$ contains no rainbow copy of $H^v$. Hence, we obtain a contradiction. 
    
    Therefore, we may (and will) assume that, for every non-cut vertex $u$ in $H$, the rainbow copy $H_\R$ contains a vertex labelled $u$. In particular, $v$ is a cut vertex in $H$. Note that $x$ and $y$ are in the same component of $H^v$ since $xy\in E(H)$ and $v \in V(H)\setminus\{x,y\}$. Let $C$ be a component of $H^v$ not containing $\{x,y\}$. Take $u \in C$ of maximum distance from $v$. Note that $u$ is not a cut vertex of $H$, and thus, by assumption, there is a vertex $u'$ in $H_\R$ labelled $u$.
    Since $H_\R$ is connected and contains at least one vertex in $M$, there is a path in $H_\R$ from $M$ to $u'$. However, by our choice of $C$, any such path passes through a vertex labelled $v$, which is a contradiction. 
\end{proof}

Recall that $T$ is the set of edges in $H$ with one endpoint in $S = N_H(x) \cap N_H(y)$ and the other in $\{x,y\}$, and that $T'$ is the set of edges in $H'$ between $\xy$ and $S$. Note that $|S| = |T'| = \frac{|T|}{2}\geq 1$ since $xy$ is contained in a triangle. By Claim \ref{claim: use every vertex}, we see that $|H_\R \cap M| \le 2$.  We will obtain a contradiction by counting the edges in $H_\R$ (via two cases). 

First, consider the case where $|H_\R \cap M| = 2$. By Claim~\ref{claim: use every vertex}, $H_\R$ contains exactly one vertex labelled $v$, for each $v \in V(H)\setminus\{x,y\}$. Hence there are at most $2|T'|$ edges in $H_\R$ with a label in $T'$. Now, consider an edge $uv \in E(H') \setminus T'$, where $u \neq \xy$ without loss of generality. Since there is exactly one vertex labelled $u$, all edges labelled $uv$ belong to the same copy of $F$, and thus have the same colour. It follows that the edges of $H_\R$ are coloured using at most $|E(H')| - |T'| + 2|T'| \leq |E(H)| - 1$ colours, which contradicts the assumption that $H_\R$ is a rainbow copy of $H$.

Now, consider the case where $|H_\R \cap M| = 1$. In this case, by Claim~\ref{claim: use every vertex} there is exactly one $v \in V(H)\setminus\{x,y\}$ that labels two vertices in $H_\R$, and every other vertex in $V(H')$ appears once as a label. 
We claim that for every edge $e\in E(H')$, except possibly for $v \xy$ (if this is an edge), all edges labelled $e$ have the same colour. Indeed, if $e$ does not contain $v$ then the claim clearly follows, as there is at most one edge labelled $e$. Otherwise, $e = uv$ with $u \neq \xy$, and thus, since $e \in E(H') \setminus T'$ and all edges labelled $e$ are in the same copy of $F$, they all receive the same colour. 
Note that, if $v \xy$ is an edge in $H'$, then there are at most two edges in $H_\R$ with this label. Altogether, the edges of $H_\R$ see at most $|E(H')| + 1 = |E(H)| - |T'| < |E(H)|$ colours, which contradicts $H_\R$ being a rainbow copy of $H$.
\end{proof}

\section{Rainbow saturation for cliques}
\label{sec:Cliques}

The purpose of this section is to prove Theorem~\ref{clique}, restated here for convenience.

\thmClique*

We begin by proving the following lemma.

\begin{lem} \label{claim:s}
	For every $s \ge 3$, there is an edge-coloured graph $(H,\chi_H)$ on $\binom{s}{2} + s$ vertices with the following properties:
	\begin{enumerate}[label = \rm(\arabic*)]
		\item \label{itm:a}
			the largest rainbow clique in $H$ has $\binom{s}{2}+1$ vertices,
		\item \label{itm:b}
			for every $v \in V(H)$, there is a rainbow clique in $H \setminus \{v\}$ on $\binom{s}{2}+1$ vertices,
		\item \label{itm:c}
			for every colour $c$, there is a rainbow clique with $\binom{s}{2}+1$ vertices containing no edge coloured $c$. 
	\end{enumerate}
\end{lem}

\begin{proof}
	Let $S$ be a set of size $s$. Let $H$ be a complete graph on vertices $S \cup S^{(2)}$, where $S^{(2)}=\{\{x,y\} \subseteq S : x \neq y\}$ is the set of pairs of distinct elements from $S$. For each pair $\{x, y\} \in S^{(2)}$, colour the edges $(x, \{x, y\})$ and $(y, \{x, y\})$ with the same colour, picking a different colour for each pair. The remaining edges receive a rainbow colouring, with colours distinct from those used so far. 
	
	We first show that the largest rainbow clique in $H$ has size $\binom{s}{2}+1$. Suppose that $T$ is a rainbow clique, and write $t := |T \cap S|$. For every two distinct vertices $x, y \in T \cap S$, the pair $\{x, y\}$ is not in $T$, since the edges $(x, \{x,y\})$ and $(y, \{x, y\})$ have the same colour. It follows that $|T| \le t + \binom{s}{2} - \binom{t}{2} \le \binom{s}{2} + 1$, using the fact that $\binom{t}{2} \ge t-1$ for $t \ge 0$. This shows that $H$ has no rainbow cliques of size at least $\binom{s}{2} + 2$. Given distinct vertices $x, y \in S$, let  $T(x, y) := \{x, y\} \cup (S^{(2)} \setminus \{\{x, y\}\})$. It is not difficult to see that $T(x, y)$ is a rainbow clique of size $\binom{s}{2} + 1$, and hence a maximum rainbow clique. This proves \ref{itm:a}.
	
	For \ref{itm:b}, let $v \in V(H)$. If $v \in S$, then for any $x, y \in S\setminus \{v\}$, the clique $T(x,y)$ avoids $v$ (using $s \ge 3$). If $v \in S^{(2)}$, then $v = \{x,y\}$ for some $x \not= y \in S$, and thus the clique $T(x,y)$ avoids $v$. This proves \ref{itm:b}. 
	
	For \ref{itm:c}, let $c$ be a colour that is used on $H$ (if $c$ does not appear on $H$, then \ref{itm:c} holds trivially). Since every colour class is either an edge or a path of length $2$, there is a vertex $v$ which is incident with all edges of colour $c$. By \ref{itm:b}, there is a maximum rainbow clique avoiding $v$ and thus $c$. This proves \ref{itm:c}.
\end{proof}

We now complete the proof of Theorem~\ref{clique}. 

\begin{proof}[Proof of Theorem~\ref{clique}]
	Given $r \ge 10$, let $s$ be the largest integer such that $r \ge \binom{s}{2} + 1$ (so $r \le \binom{s+1}{2}$). Write $t := r - \binom{s}{2} - 1$. Notice that we may take $n$ to be sufficiently large.
	
	Let $(H,\chi_H)$ be an edge coloured graph as in \Cref{claim:s} (applied with $s$), and let $T_1$, $T_2$ and $U$ be pairwise disjoint sets of sizes $t$, $t$, and $n - 2t - |V(H)|$ respectively which are disjoint from $V(H)$. 
	Let $G$ be a graph on vertex set $V(H) \cup U \cup T_1 \cup T_2$ with edge set $E(G):=\{uv: \{u,v\} \not\subseteq U \text{ and } uv \notin T_1 \times T_2\}$. We define a colouring $\chi$ on $E(G)$ such that $\chi(e) = \chi_H(e)$ for each $e \in E(H)$, and the edges of $E(G)\setminus E(H)$ receive a rainbow colouring with colours distinct from those used by $\chi_H$. Note that the vertices of each $T_i$ induce a rainbow clique of size $t$.
	
	We claim that $G$ contains no rainbow copy of $K_{r+2}$. Indeed, suppose $W$ is a rainbow clique in $G$. Then $W$ contains at most one vertex from $U$ (as $U$ is an independent set), at most $t$ vertices from $T_1 \cup T_2$ (as no edges exist between $T_1$ and $T_2$), and at most $\binom{s}{2} + 1$ vertices from $H$ (by the choice of $\chi_H$). Thus, $|W| \le \binom{s}{2} + 1 + t + 1 = r+1$, as claimed.
	
	Next, we show that the addition of any non-edge $xy$ in $U$ in any colour creates a rainbow copy of $K_{r+2}$. Suppose that $xy$ is coloured red, and denote by $G'$ the graph obtained from $G$ by adding $xy$ in red. 
	Then exactly one of the following holds: no edge in $G[V(H) \cup \{x,y\}]$ is red; exactly one of the edges from $\{x, y\}$ to $V(H)$ is red; red is used in $H$. 
	Hence, by the properties of $H$ given by \Cref{claim:s} and by definition of $G$, there is a maximum rainbow clique $W$ in $H$ such that $G[W \cup \{x, y\}]$ does not contain any red edges. Thus $W \cup \{x, y\}$ is a rainbow clique of size $\binom{s}{2} + 3 = r-t+2$ in $G'$. Since either $T_1$ or $T_2$ is not incident to any red edge, there is $i \in \{1,2\}$ for which $W \cup T_i \cup \{x, y\}$ is a rainbow clique of size $r+2$, as required.
	
	Now, as $r \le \binom{s+1}{2}$, we have $t \le \binom{s+1}{2} - \binom{s}{2} - 1 = s - 1$. Also, since $r \ge \binom{s}{2}$, we have $s - 1 \le \sqrt{2r}$. Thus, thinking of $r$ as fixed, we have
	\begin{align*}
		|E(G)| & = (|V(H)| + 2t)(n - |V(H)| - 2t) + |E(H)|+ 2\binom{t}{2} + 2t|V(H)| \\
		& = (|V(H)| + 2t)(n - |V(H)| - 2t) + O(1) \\
		& = \left(\binom{s}{2} + s + 2t\right) \cdot n + O(1) \\
		& = (r + s + t - 1) \cdot n + O(1)\\
        & \le (r + 2s - 2) \cdot n + O(1) \\ 
        & \le (r + 2\sqrt{2r}) \cdot n + O(1).
    \end{align*}
    Observe that $G$ has at most $t^2 = O(1)$ bad non-edges, namely the edges $uv$ such that $u \in T_1$ and $v \in T_2$, so applying \Cref{absorb} yields that $\rsat(n,K_{r+2}) \le (r + 2\sqrt{2r}) n + O(1).$
\end{proof}

\section{Conclusion and open problems}
\label{sec:Future}

Our main theorem shows that for any graph $H$ the rainbow saturation number $\rsat(n,H)$ is linear in $n$. Given this, it is natural to ask the following.

\begin{qn}\label{ques:lim}
    For every graph $H$, does there exist a constant $c = c(H)$ such that
    $$\rsat(n,H) = (c(H) + o(1))n?$$
\end{qn}

Analogous questions have been considered for saturation and the related notion of weak saturation. The limit $\lim_{n \rightarrow \infty} \frac{\sat(n,H)}{n}$ was conjectured to exist by Tuza; see~\cite{Tuza,TuzaExtremal}. Although some progress has been made towards Tuza's conjecture (see~\cite{Pikhurko,TT}), it remains open. 

Another natural direction would be to generalise the notion of weak saturation to edge-coloured graphs. A subgraph $G$ of $F$ is said to be \emph{weakly $(F, H)$-saturated} if the edges of $E(F) \setminus E(G)$ can be added to $G$, one edge at a time, in such a way that every added edge creates a new copy of $H$.  The minimum number of edges in a weakly $(F, H)$-saturated graph is known as the \emph{weak saturation number}, and is denoted $\wsat(F, H)$. When $F=K_n$, we write $\wsat(n, H)$ instead of $\wsat(K_n, H)$.  

We say an edge-coloured subgraph $(G,\chi)$ of $F$ is \emph{weakly} $H$\emph{-rainbow saturated} if there exists an ordering $e_1,...,e_m$ of $E(F)\setminus E(G)$ such that, for any list $c_1,...,c_m$ of distinct colours from $\N$, the non-edges $e_i$ in colour $c_i$ can be added to $G$, one at a time, so that every added edge creates a new rainbow copy of $H$. The \emph{weak rainbow saturation number} of $H$, denoted by $\rwsat(F,H)$, is the minimum number of edges in a weakly $H$-rainbow saturated graph. When $F=K_n$, we write $\rwsat(n, H)$ instead of $\rwsat(K_n, H)$. 

Note that we require the collection of added edges to receive distinct colours. In particular, we wish to exclude the possibility that all added edges have the same colour, in which case the previously added edges do not contribute to making new rainbow copies and the problem reduces to the standard rainbow saturation number.

The study of weak saturation numbers was introduced by Bollob\'{a}s~\cite{Bollobas} in 1968, where he proved that $\sat(n,K_m) = \wsat(n,K_m)$ for all $3 \le m \le 7$ and conjectured that this equality holds for all $m$. This conjecture was first proved by Lov\'{a}sz~\cite{Lovasz} using a beautiful generalisation of the Bollob\'{a}s two families theorem. It would be interesting to determine whether a similar phenomenon holds for the rainbow saturation number. 

\begin{qn}
Let $r \ge 3$. Is $\rwsat(n,K_r) = \rsat(n,K_r)?$
\end{qn}

It would be interesting to see if  upper bounds on $\rwsat(n,K_r)$ can be found that are much stronger than those provided by Theorem~\ref{clique} (note that $\rsat(n,K_r) \ge \rwsat(n,K_r)$).  
\begin{qn}
Let $r \ge 3$. Is it the case that $\rwsat(n,K_r) \le rn + O(1)$, for $n$ sufficiently large? 
\end{qn}

Alon~\cite{Alon} proved in 1985 that $\wsat(n, H) = (c(H) + o(1))n$, for all graphs $H$. The natural generalisation of this to hypergraphs was recently proved by Shapira and Tyomkin~\cite{ST}. An analogous question can be considered regarding the weak rainbow saturation number.

\begin{qn}\label{ques:lim2}
    For every graph $H$, does there exist a constant $c = c(H)$ such that
    $$\rwsat(n,H) = (c(H) + o(1))n?$$
\end{qn}

In the above definition of weak rainbow saturation we do not require $(G,\chi)$ to be rainbow $H$-free. Note however that one could alternatively consider the minimum number of edges in a weakly $H$-rainbow saturated, rainbow $H$-free graph. Unlike with weak saturation numbers, it is not clear whether this gives the same value as the weak rainbow saturation number of $H$. Indeed, a graph may contain a rainbow copy of $H$ and still be a minimal (with respect to edge removal) weakly $H$-rainbow saturated graph. For example, a rainbow $K_n$ is weakly $K_n$-saturated but removing any edge does not leave a weakly $K_n$-rainbow saturated graph, due to the stipulation that adding the edge back in in any colour must create a rainbow $K_n$. 

\subsection{Proper rainbow saturation}
As our construction does not give a proper colouring, it is natural to ask what happens if we restrict our attention to properly edge-coloured graphs. We make the following definition:
A properly edge-coloured graph $(G,c)$ is \emph{properly} $H$\emph{-rainbow saturated} if $(G,c)$ does not contain a rainbow copy of $H$, but the addition of any non-edge, in any colour from $\N$ which preserves the proper colouring, creates a rainbow copy of $H$. The \emph{proper rainbow saturation number} of $H$, denoted by $\prsat(n,H)$, is the minimum number of edges in a properly $H$-rainbow saturated edge-coloured graph on $n$ vertices.  

\begin{qn}
    Is $\prsat(n,H) \le \rsat(n,H)$ for all $H$?
\end{qn}

It is worth noting that the phrase `rainbow saturation' has appeared in the literature in a different context. Recently, Bushaw, Johnston, and Rombach~\cite{BJR} defined a different form of rainbow saturation which also requires the colouring to be proper. We will refer to this concept as \emph{BJR proper rainbow saturation} to distinguish it from the definition above. 

A graph $G$ is \emph{BJR properly} $H$\emph{-rainbow saturated} if there is a proper colouring of $G$ that does not contain a rainbow copy of $H$, but if any non-edge is added to $G$ then any proper colouring contains a rainbow copy of $H$. The \emph{BJR proper rainbow saturation number} of $H$, denoted by $\prsat'(n,H)$, is the minimum number of edges in a properly $H$-rainbow saturated graph on $n$ vertices.  

Note that the two definitions are subtly different. Any graph $G$ that is BJR properly $H$-rainbow saturated gives rise to a properly $H$-rainbow saturated edge-coloured graph $(G,c)$ by taking $c$ to be the proper colouring of $G$ that does not contain a rainbow copy of $H$. This observation tells us that $\prsat(n,H) \le \prsat'(n,H)$ for all $n$ and $H$.
However, the converse might not hold: if $(G,c)$ is properly $H$-rainbow saturated, there is no guarantee that every recolouring of $G$ plus a non-edge contains a rainbow $H$.

Bushaw, Johnston, and Rombach~\cite{BJR} proved that, for any graph $H$ that does not include an induced even cycle, the BJR proper rainbow saturation number of $H$ is linear in $n$, that is, $\prsat'(n, H) = O(n)$.
They also showed that $\prsat'(n, C_4)$ is again linear in $n$ and conjecture that, analogously to the classical saturation numbers, the proper rainbow saturation number is linear in $n$ for every $H$.
Private correspondence with Bushaw, Johnston and Rombach and independently with Barnab\'{a}s Janzer informs us that this can be shown by a straightforward application of a result of Kazsonyi and Tuza \cite{KT}. It follows that $\prsat(n,H)$ must also be linear in $n$ for all $H$. It would be interesting to know if there are any graphs $H$ where $\prsat(n,H)$ and $\prsat'(n,H)$ differ considerably.

\paragraph{Acknowledgements}
We would like to thank the anonymous referees for their helpful comments.
\bibliographystyle{abbrv}
\bibliography{rsat} 
\label{Refs}

\end{document}

%% file: fig1.tex
\begin{tikzpicture}[scale=\textwidth/12cm]
    \draw[fill] (0,1) circle (.05);
    \draw[fill] (-0.577, 0) circle (.05);
    \draw[fill] (0.577, 0) circle (.05);
    \draw (0, 1) -- (0.577, 0) -- (-0.577, 0) -- (0,1);

    \draw[decorate, decoration={brace, amplitude = 7}] (-0.637, 1.1) -- (0.637, 1.1);
    \node[above] at (0, 1.25) {$H'$};

    \begin{scope}[shift={(1.454, 0)}]
        \draw[fill] (0,1) circle (.05);
        \draw[fill] (-0.577, 0) circle (.05);
        \draw[fill] (0.577, 0) circle (.05);
        \draw (0, 1) -- (0.577, 0) -- (-0.577, 0) -- (0,1);
    \end{scope}

    \begin{scope}[shift={(2.908, 0)}]
        \draw[fill] (0,1) circle (.05);
        \draw[fill] (-0.577, 0) circle (.05);
        \draw[fill] (0.577, 0) circle (.05);
        \draw (0, 1) -- (0.577, 0) -- (-0.577, 0) -- (0,1);
    \end{scope}

    \begin{scope}[shift={(3.785, 0)}]
        \draw[fill] (0,1) circle (.05);
        \draw[fill] (0, 0) circle (.05);
        \draw (0, 1) -- (0,0);
    \end{scope}

    \begin{scope}[shift={(4.085, 0)}]
        \draw[fill] (0,1) circle (.05);
        \draw[fill] (0, 0) circle (.05);
        \draw (0, 1) -- (0,0);
    \end{scope}

    \draw[decorate, decoration={brace, mirror, amplitude = 7}] (-0.637, -0.1) -- (4.135, -0.1);
    \node[below] at (1.749, -0.25) {$H$};

    \begin{scope}[shift={(6.085, 0)}]
        \draw[fill] (0,1) circle (.05);
        \draw[fill] (-0.577, 0) circle (.05);
        \draw[fill] (0.577, 0) circle (.05);
        \draw[fill] (1.154,1) circle (.05);
        \draw[fill] (1.731,0) circle (.05);
        \draw (0, 1) -- (0.577, 0) -- (-0.577, 0) -- (0,1);
        \draw (1.154, 1) -- (0.577, 0) -- (1.731, 0) -- (1.154,1);
    \end{scope}

    \draw[decorate, decoration={brace, mirror, amplitude = 7}] (5.458, -0.1) -- (7.866, -0.1);
    \node[below] at (6.662, -0.25) {$H_2$};

    \begin{scope}[shift={(0.5135, -2.5)}]

    \begin{scope}[shift={(0, 0)}]
        \draw[fill] (0,1) circle (.05);
        \draw[fill] (-0.577, 0) circle (.05);
        \draw[fill] (0.577, 0) circle (.05);
        \draw[fill] (1.154,1) circle (.05);
        \draw[fill] (1.731,0) circle (.05);
        \draw (0, 1) -- (0.577, 0) -- (-0.577, 0) -- (0,1);
        \draw (1.154, 1) -- (0.577, 0) -- (1.731, 0) -- (1.154,1);
    \end{scope}

    \begin{scope}[shift={(2.608, 0)}]
        \draw[fill] (0,1) circle (.05);
        \draw[fill] (-0.577, 0) circle (.05);
        \draw[fill] (0.577, 0) circle (.05);
        \draw[fill] (1.154,1) circle (.05);
        \draw[fill] (1.731,0) circle (.05);
        \draw (0, 1) -- (0.577, 0) -- (-0.577, 0) -- (0,1);
        \draw (1.154, 1) -- (0.577, 0) -- (1.731, 0) -- (1.154,1);
    \end{scope}

    \begin{scope}[shift={(4.639, 0)}]
        \draw[fill] (0,1) circle (.05);
        \draw[fill] (0, 0) circle (.05);
        \draw (0, 1) -- (0,0);
    \end{scope}

    \begin{scope}[shift={(4.939, 0)}]
        \draw[fill] (0,1) circle (.05);
        \draw[fill] (0, 0) circle (.05);
        \draw (0, 1) -- (0,0);
    \end{scope}

    \begin{scope}[shift={(5.239, 0)}]
        \draw[fill] (0,1) circle (.05);
        \draw[fill] (0, 0) circle (.05);
        \draw (0, 1) -- (0,0);
    \end{scope}

    \begin{scope}[shift={(5.539, 0)}]
        \draw[fill] (0,1) circle (.05);
        \draw[fill] (0, 0) circle (.05);
        \draw (0, 1) -- (0,0);
    \end{scope}

    \draw (6.339, 0.5) circle (.5);
    \node at (6.339, 0.5) {$G$};

    \draw[decorate, decoration={brace, mirror, amplitude = 7}] (-0.637, -0.1) -- (6.839, -0.1);
    \node[below] at (3.101, -0.25) {$G^*$};

\end{scope}

\end{tikzpicture}

%% file: fig2.tex
\begin{tikzpicture}[scale=\textwidth/12cm]

    \begin{scope}[shift={(-2, 0)}]

        \draw[teal] (-0.95, 0.31) -- (0,1);
        \draw[cyan] (0.95, 0.31) -- (0,1);
        \draw[orange] (0, 0.31) -- (0,1);
        \draw[blue] (-0.95, 0.31) -- (-0.59,-0.81);
        \draw[red] (0.95, 0.31) -- (0.59,-0.81);
        \draw[violet] (0.59,-0.81) -- (0,0.31);
        \draw[olive] (-0.59,-0.81) -- (0,0.31);

        \draw[black] (-0.59,-0.81) -- (0.59,-0.81);

        \draw[fill] (0,1) circle (.05);
        \draw[fill] (-0.95, 0.31) circle (.05);
        \draw[fill] (0.95, 0.31) circle (.05);
        \draw[fill] (0, 0.31) circle (.05);
        \draw[fill] (0.59,-0.81) circle (.05);
        \draw[fill] (-0.59,-0.81) circle (.05);

        \node[below = 0.1cm] at (-0.59,-0.81) {$x$};
        \node[below = 0.1cm] at (0.59,-0.81) {$y$};

        \node[above left] at (-0.475, 0.655) {$e_1$};
        \node[right] at (0, 0.55) {$e_2$};
        \node[above right] at (0.475, 0.655) {$e_3$};
        \node[below left] at (-0.77, -0.25) {$e_4$};
        \node[below right] at (0.77, -0.25) {$e_5$};

        \draw[decorate, decoration={brace, amplitude = 7}] (-0.95, 1.1) -- (0.95, 1.1);
        \node[above] at (0, 1.25) {$H$};
    \end{scope}

    \begin{scope}[shift={(2, 0)}]

        \draw[teal] (-0.95, 0.31) -- (0,1);
        \draw[cyan] (0.95, 0.31) -- (0,1);
        \draw[orange] (0, 0.31) -- (0,1);
        \draw[blue] (-0.95, 0.31) -- (0,-0.81);
        \draw[red] (0.95, 0.31) -- (0,-0.81);
        \draw[dashed] (0, 0.31) -- (0,-0.81);

        \draw[fill] (0,1) circle (.05);
        \draw[fill] (-0.95, 0.31) circle (.05);
        \draw[fill] (0.95, 0.31) circle (.05);
        \draw[fill] (0, 0.31) circle (.05);
        \draw[fill] (0,-0.81) circle (.05);

        \node[below = 0.1cm] at (0,-0.81) {$\xy$};
        \draw[decorate, decoration={brace, amplitude = 7}] (-0.95, 1.1) -- (0.95, 1.1);
        \node[above] at (0, 1.25) {$H'$};
    \end{scope}

    \begin{scope}[shift={(-4.5, -3.5)}]

        \draw[gray]  (0,-0.81) ellipse (0.9 and 0.2);

        \node[left] at (-1.02, -0.81) {$M$};
       \draw[decorate, decoration={brace, mirror, amplitude = 3}] (-0.98, -0.66) -- (-0.98, -0.96);

        \draw[black] (-0.95, 0.31) -- (0,1);
        \draw[cyan] (0.95, 0.31) -- (0,1);
        \draw[orange] (0, 0.31) -- (0,1);
        \draw[blue] (-0.95, 0.31) -- (-0.59,-0.81);
        \draw[blue] (-0.95, 0.31) -- (-0.39,-0.81);
        \draw[blue] (-0.95, 0.31) -- (0.39,-0.81);
        \draw[blue] (-0.95, 0.31) -- (0.59,-0.81);
        \draw[red] (0.95, 0.31) -- (0.59,-0.81);
        \draw[red] (0.95, 0.31) -- (0.39,-0.81);
        \draw[red] (0.95, 0.31) -- (-0.39,-0.81);
        \draw[red] (0.95, 0.31) -- (-0.59,-0.81);

        \draw[dashed] (0, 0.31) -- (0.59,-0.81);
        \draw[dashed] (0, 0.31) -- (0.39,-0.81);
        \draw[dashed] (0, 0.31) -- (-0.39,-0.81);
        \draw[dashed] (0, 0.31) -- (-0.59,-0.81);

        \draw[fill] (0,1) circle (.05);
        \draw[fill] (-0.95, 0.31) circle (.05);
        \draw[fill] (0.95, 0.31) circle (.05);
        \draw[fill] (0, 0.31) circle (.05);
        \draw[fill] (0.59,-0.81) circle (.05);
        \draw[fill] (0.39,-0.81) circle (.05);
        \draw[fill] (-0.39,-0.81) circle (.05);
        \draw[fill] (-0.59,-0.81) circle (.05);

        \node at (0, -0.81) {\large$\cdots$};

        \draw[decorate, decoration={brace, amplitude = 7}] (-0.95, 1.1) -- (0.95, 1.1);
        \node[above] at (0, 1.25) {$F_{e_1}$};

    \end{scope}

    \begin{scope}[shift={(-1.5, -3.5)}]

        \draw[gray]  (0,-0.81) ellipse (0.9 and 0.2);

        \draw[teal] (-0.95, 0.31) -- (0,1);
        \draw[cyan] (0.95, 0.31) -- (0,1);
        \draw[black] (0, 0.31) -- (0,1);
        \draw[blue] (-0.95, 0.31) -- (-0.59,-0.81);
        \draw[blue] (-0.95, 0.31) -- (-0.39,-0.81);
        \draw[blue] (-0.95, 0.31) -- (0.39,-0.81);
        \draw[blue] (-0.95, 0.31) -- (0.59,-0.81);
        \draw[red] (0.95, 0.31) -- (0.59,-0.81);
        \draw[red] (0.95, 0.31) -- (0.39,-0.81);
        \draw[red] (0.95, 0.31) -- (-0.39,-0.81);
        \draw[red] (0.95, 0.31) -- (-0.59,-0.81);

        \draw[dashed] (0, 0.31) -- (0.59,-0.81);
        \draw[dashed] (0, 0.31) -- (0.39,-0.81);
        \draw[dashed] (0, 0.31) -- (-0.39,-0.81);
        \draw[dashed] (0, 0.31) -- (-0.59,-0.81);

        \draw[fill] (0,1) circle (.05);
        \draw[fill] (-0.95, 0.31) circle (.05);
        \draw[fill] (0.95, 0.31) circle (.05);
        \draw[fill] (0, 0.31) circle (.05);
        \draw[fill] (0.59,-0.81) circle (.05);
        \draw[fill] (0.39,-0.81) circle (.05);
        \draw[fill] (-0.39,-0.81) circle (.05);
        \draw[fill] (-0.59,-0.81) circle (.05);

        \node at (0, -0.81) {\large$\cdots$};

        \draw[decorate, decoration={brace, amplitude = 7}] (-0.95, 1.1) -- (0.95, 1.1);
        \node[above] at (0, 1.25) {$F_{e_2}$};
    \end{scope}

    \begin{scope}[shift={(1.5, -3.5)}]

        \draw[gray]  (0,-0.81) ellipse (0.9 and 0.2);

        \draw[teal] (-0.95, 0.31) -- (0,1);
        \draw[cyan] (0.95, 0.31) -- (0,1);
        \draw[orange] (0, 0.31) -- (0,1);
        \draw[blue] (-0.95, 0.31) -- (-0.59,-0.81);
        \draw[blue] (-0.95, 0.31) -- (-0.39,-0.81);
        \draw[blue] (-0.95, 0.31) -- (0.39,-0.81);
        \draw[blue] (-0.95, 0.31) -- (0.59,-0.81);
        \draw[black] (0.95, 0.31) -- (0.59,-0.81);
        \draw[black] (0.95, 0.31) -- (0.39,-0.81);
        \draw[black] (0.95, 0.31) -- (-0.39,-0.81);
        \draw[black] (0.95, 0.31) -- (-0.59,-0.81);

        \draw[dashed] (0, 0.31) -- (0.59,-0.81);
        \draw[dashed] (0, 0.31) -- (0.39,-0.81);
        \draw[dashed] (0, 0.31) -- (-0.39,-0.81);
        \draw[dashed] (0, 0.31) -- (-0.59,-0.81);

        \draw[fill] (0,1) circle (.05);
        \draw[fill] (-0.95, 0.31) circle (.05);
        \draw[fill] (0.95, 0.31) circle (.05);
        \draw[fill] (0, 0.31) circle (.05);
        \draw[fill] (0.59,-0.81) circle (.05);
        \draw[fill] (0.39,-0.81) circle (.05);
        \draw[fill] (-0.39,-0.81) circle (.05);
        \draw[fill] (-0.59,-0.81) circle (.05);

        \node at (0, -0.81) {\large$\cdots$};
        \draw[decorate, decoration={brace, amplitude = 7}] (-0.95, 1.1) -- (0.95, 1.1);
        \node[above] at (0, 1.25) {$F_{e_5}$};
    \end{scope}

    \begin{scope}[shift={(4.5, -3.5)}]

        \draw[gray]  (0,-0.81) ellipse (0.9 and 0.2);

        \draw[teal] (-0.95, 0.31) -- (0,1);
        \draw[cyan] (0.95, 0.31) -- (0,1);
        \draw[orange] (0, 0.31) -- (0,1);
        \draw[blue] (-0.95, 0.31) -- (-0.59,-0.81);
        \draw[blue] (-0.95, 0.31) -- (-0.39,-0.81);
        \draw[blue] (-0.95, 0.31) -- (0.39,-0.81);
        \draw[blue] (-0.95, 0.31) -- (0.59,-0.81);
        \draw[red] (0.95, 0.31) -- (0.59,-0.81);
        \draw[red] (0.95, 0.31) -- (0.39,-0.81);
        \draw[red] (0.95, 0.31) -- (-0.39,-0.81);
        \draw[red] (0.95, 0.31) -- (-0.59,-0.81);

        \draw[dashed] (0, 0.31) -- (0.59,-0.81);
        \draw[dashed] (0, 0.31) -- (0.39,-0.81);
        \draw[dashed] (0, 0.31) -- (-0.39,-0.81);
        \draw[dashed] (0, 0.31) -- (-0.59,-0.81);

        \draw[fill] (0,1) circle (.05);
        \draw[fill] (-0.95, 0.31) circle (.05);
        \draw[fill] (0.95, 0.31) circle (.05);
        \draw[fill] (0, 0.31) circle (.05);
        \draw[fill] (0.59,-0.81) circle (.05);
        \draw[fill] (0.39,-0.81) circle (.05);
        \draw[fill] (-0.39,-0.81) circle (.05);
        \draw[fill] (-0.59,-0.81) circle (.05);

        \node at (0, -0.81) {\large$\cdots$};
        \draw[decorate, decoration={brace, amplitude = 7}] (-0.95, 1.1) -- (0.95, 1.1);
        \node[above] at (0, 1.25) {$F_{1}$};
    \end{scope}

    \begin{scope}[shift={(-0.75, -6.5)}]

        \draw[teal] (-0.475, 0.155) -- (0,0.5);
        \draw[cyan] (0.475, 0.155) -- (0,0.5);
        \draw[orange] (0, 0.155) -- (0,0.5);

        \draw[black] (-0.475, 0.155) -- +(-45:0.25);
        \draw[black] (-0.475, 0.155) -- +(-75:0.25);
        \draw[black] (-0.475, 0.155) -- +(-105:0.25);
        \draw[black] (-0.475, 0.155) -- +(-135:0.25);

        \draw[dashed] (-0, 0.155) -- +(-45:0.25);
        \draw[dashed] (-0, 0.155) -- +(-75:0.25);
        \draw[dashed] (-0, 0.155) -- +(-105:0.25);
        \draw[dashed] (-0, 0.155) -- +(-135:0.25);

        \draw[red] (0.475, 0.155) -- +(-45:0.25);
        \draw[red] (0.475, 0.155) -- +(-75:0.25);
        \draw[red] (0.475, 0.155) -- +(-105:0.25);
        \draw[red] (0.475, 0.155) -- +(-135:0.25);

        \draw[fill] (0,0.5) circle (.05);
        \draw[fill] (-0.475, 0.155) circle (.05);
        \draw[fill] (0.475, 0.155) circle (.05);
        \draw[fill] (0, 0.155) circle (.05);

        \begin{scope}[shift={(-4.5, 0)}]

            \draw[black] (-0.475, 0.155) -- (0,0.5);
            \draw[cyan] (0.475, 0.155) -- (0,0.5);
            \draw[orange] (0, 0.155) -- (0,0.5);

            \draw[blue] (-0.475, 0.155) -- +(-45:0.25);
            \draw[blue] (-0.475, 0.155) -- +(-75:0.25);
            \draw[blue] (-0.475, 0.155) -- +(-105:0.25);
            \draw[blue] (-0.475, 0.155) -- +(-135:0.25);

            \draw[dashed] (-0, 0.155) -- +(-45:0.25);
            \draw[dashed] (-0, 0.155) -- +(-75:0.25);
            \draw[dashed] (-0, 0.155) -- +(-105:0.25);
            \draw[dashed] (-0, 0.155) -- +(-135:0.25);

            \draw[red] (0.475, 0.155) -- +(-45:0.25);
            \draw[red] (0.475, 0.155) -- +(-75:0.25);
            \draw[red] (0.475, 0.155) -- +(-105:0.25);
            \draw[red] (0.475, 0.155) -- +(-135:0.25);

            \draw[fill] (0,0.5) circle (.05);
            \draw[fill] (-0.475, 0.155) circle (.05);
            \draw[fill] (0.475, 0.155) circle (.05);
            \draw[fill] (0, 0.155) circle (.05);

        \end{scope}

        \begin{scope}[shift={(-3, 0)}]

            \draw[teal] (-0.475, 0.155) -- (0,0.5);
            \draw[cyan] (0.475, 0.155) -- (0,0.5);
            \draw[black] (0, 0.155) -- (0,0.5);

            \draw[blue] (-0.475, 0.155) -- +(-45:0.25);
            \draw[blue] (-0.475, 0.155) -- +(-75:0.25);
            \draw[blue] (-0.475, 0.155) -- +(-105:0.25);
            \draw[blue] (-0.475, 0.155) -- +(-135:0.25);

            \draw[dashed] (-0, 0.155) -- +(-45:0.25);
            \draw[dashed] (-0, 0.155) -- +(-75:0.25);
            \draw[dashed] (-0, 0.155) -- +(-105:0.25);
            \draw[dashed] (-0, 0.155) -- +(-135:0.25);

            \draw[red] (0.475, 0.155) -- +(-45:0.25);
            \draw[red] (0.475, 0.155) -- +(-75:0.25);
            \draw[red] (0.475, 0.155) -- +(-105:0.25);
            \draw[red] (0.475, 0.155) -- +(-135:0.25);

            \draw[fill] (0,0.5) circle (.05);
            \draw[fill] (-0.475, 0.155) circle (.05);
            \draw[fill] (0.475, 0.155) circle (.05);
            \draw[fill] (0, 0.155) circle (.05);

        \end{scope}

        \begin{scope}[shift={(-1.5, 0)}]

            \draw[teal] (-0.475, 0.155) -- (0,0.5);
            \draw[black] (0.475, 0.155) -- (0,0.5);
            \draw[orange] (0, 0.155) -- (0,0.5);

            \draw[blue] (-0.475, 0.155) -- +(-45:0.25);
            \draw[blue] (-0.475, 0.155) -- +(-75:0.25);
            \draw[blue] (-0.475, 0.155) -- +(-105:0.25);
            \draw[blue] (-0.475, 0.155) -- +(-135:0.25);

            \draw[dashed] (-0, 0.155) -- +(-45:0.25);
            \draw[dashed] (-0, 0.155) -- +(-75:0.25);
            \draw[dashed] (-0, 0.155) -- +(-105:0.25);
            \draw[dashed] (-0, 0.155) -- +(-135:0.25);

            \draw[red] (0.475, 0.155) -- +(-45:0.25);
            \draw[red] (0.475, 0.155) -- +(-75:0.25);
            \draw[red] (0.475, 0.155) -- +(-105:0.25);
            \draw[red] (0.475, 0.155) -- +(-135:0.25);

            \draw[fill] (0,0.5) circle (.05);
            \draw[fill] (-0.475, 0.155) circle (.05);
            \draw[fill] (0.475, 0.155) circle (.05);
            \draw[fill] (0, 0.155) circle (.05);

        \end{scope}

        \begin{scope}[shift={(1.5, 0)}]

            \draw[teal] (-0.475, 0.155) -- (0,0.5);
            \draw[cyan] (0.475, 0.155) -- (0,0.5);
            \draw[orange] (0, 0.155) -- (0,0.5);

            \draw[blue] (-0.475, 0.155) -- +(-45:0.25);
            \draw[blue] (-0.475, 0.155) -- +(-75:0.25);
            \draw[blue] (-0.475, 0.155) -- +(-105:0.25);
            \draw[blue] (-0.475, 0.155) -- +(-135:0.25);

            \draw[dashed] (-0, 0.155) -- +(-45:0.25);
            \draw[dashed] (-0, 0.155) -- +(-75:0.25);
            \draw[dashed] (-0, 0.155) -- +(-105:0.25);
            \draw[dashed] (-0, 0.155) -- +(-135:0.25);

            \draw[black] (0.475, 0.155) -- +(-45:0.25);
            \draw[black] (0.475, 0.155) -- +(-75:0.25);
            \draw[black] (0.475, 0.155) -- +(-105:0.25);
            \draw[black] (0.475, 0.155) -- +(-135:0.25);

            \draw[fill] (0,0.5) circle (.05);
            \draw[fill] (-0.475, 0.155) circle (.05);
            \draw[fill] (0.475, 0.155) circle (.05);
            \draw[fill] (0, 0.155) circle (.05);

        \end{scope}

        \begin{scope}[shift={(3, 0)}]

            \draw[teal] (-0.475, 0.155) -- (0,0.5);
            \draw[cyan] (0.475, 0.155) -- (0,0.5);
            \draw[orange] (0, 0.155) -- (0,0.5);

            \draw[blue] (-0.475, 0.155) -- +(-45:0.25);
            \draw[blue] (-0.475, 0.155) -- +(-75:0.25);
            \draw[blue] (-0.475, 0.155) -- +(-105:0.25);
            \draw[blue] (-0.475, 0.155) -- +(-135:0.25);

            \draw[dashed] (-0, 0.155) -- +(-45:0.25);
            \draw[dashed] (-0, 0.155) -- +(-75:0.25);
            \draw[dashed] (-0, 0.155) -- +(-105:0.25);
            \draw[dashed] (-0, 0.155) -- +(-135:0.25);

            \draw[red] (0.475, 0.155) -- +(-45:0.25);
            \draw[red] (0.475, 0.155) -- +(-75:0.25);
            \draw[red] (0.475, 0.155) -- +(-105:0.25);
            \draw[red] (0.475, 0.155) -- +(-135:0.25);

            \draw[fill] (0,0.5) circle (.05);
            \draw[fill] (-0.475, 0.155) circle (.05);
            \draw[fill] (0.475, 0.155) circle (.05);
            \draw[fill] (0, 0.155) circle (.05);

        \end{scope}

        \begin{scope}[shift={(6, 0)}]

            \draw[teal] (-0.475, 0.155) -- (0,0.5);
            \draw[cyan] (0.475, 0.155) -- (0,0.5);
            \draw[orange] (0, 0.155) -- (0,0.5);

            \draw[blue] (-0.475, 0.155) -- +(-45:0.25);
            \draw[blue] (-0.475, 0.155) -- +(-75:0.25);
            \draw[blue] (-0.475, 0.155) -- +(-105:0.25);
            \draw[blue] (-0.475, 0.155) -- +(-135:0.25);

            \draw[dashed] (-0, 0.155) -- +(-45:0.25);
            \draw[dashed] (-0, 0.155) -- +(-75:0.25);
            \draw[dashed] (-0, 0.155) -- +(-105:0.25);
            \draw[dashed] (-0, 0.155) -- +(-135:0.25);

            \draw[red] (0.475, 0.155) -- +(-45:0.25);
            \draw[red] (0.475, 0.155) -- +(-75:0.25);
            \draw[red] (0.475, 0.155) -- +(-105:0.25);
            \draw[red] (0.475, 0.155) -- +(-135:0.25);

            \draw[fill] (0,0.5) circle (.05);
            \draw[fill] (-0.475, 0.155) circle (.05);
            \draw[fill] (0.475, 0.155) circle (.05);
            \draw[fill] (0, 0.155) circle (.05);

        \end{scope}

        \node at (4.5, 0) {\large $\cdots$};

        \begin{scope}[shift={(0.75, 0)}]

            \draw[gray]  (0,-1.5) ellipse (2.25 and 0.2);

            \draw[red] (-2, -1.5) -- +(45:0.5);
            \draw[dashed] (-2, -1.5) -- +(60:0.5);
            \draw[blue] (-2, -1.5) -- +(75:0.5);
            \draw[black] (-2, -1.5) -- +(90:0.5);
            \draw[red] (-2, -1.5) -- +(105:0.5);
            \draw[dashed] (-2, -1.5) -- +(120:0.5);
            \draw[blue] (-2, -1.5) -- +(135:0.5);
            \draw[fill] (-2,-1.5) circle (.05);

            \draw[red] (0, -1.5) -- +(45:0.5);
            \draw[dashed] (0, -1.5) -- +(60:0.5);
            \draw[blue] (0, -1.5) -- +(75:0.5);
            \draw[black] (0, -1.5) -- +(90:0.5);
            \draw[red] (0, -1.5) -- +(105:0.5);
            \draw[dashed] (0, -1.5) -- +(120:0.5);
            \draw[blue] (0, -1.5) -- +(135:0.5);
            \draw[fill] (0,-1.5) circle (.05);

            \draw[red] (-1, -1.5) -- +(45:0.5);
            \draw[dashed] (-1, -1.5) -- +(60:0.5);
            \draw[blue] (-1, -1.5) -- +(75:0.5);
            \draw[black] (-1, -1.5) -- +(90:0.5);
            \draw[red] (-1, -1.5) -- +(105:0.5);
            \draw[dashed] (-1, -1.5) -- +(120:0.5);
            \draw[blue] (-1, -1.5) -- +(135:0.5);
            \draw[fill] (-1,-1.5) circle (.05);

            \node at (1, -1.5) {\large $\cdots$};

            \draw[red] (2, -1.5) -- +(45:0.5);
            \draw[dashed] (2, -1.5) -- +(60:0.5);
            \draw[blue] (2, -1.5) -- +(75:0.5);
            \draw[black] (2, -1.5) -- +(90:0.5);
            \draw[red] (2, -1.5) -- +(105:0.5);
            \draw[dashed] (2, -1.5) -- +(120:0.5);
            \draw[blue] (2, -1.5) -- +(135:0.5);
            \draw[fill] (2,-1.5) circle (.05);

        \end{scope}
        
        \node[above] at (-4.5, 0.5) {$F_{e_1}$};
        \node[above] at (-3, 0.5) {$F_{e_2}$};
        \node[above] at (-1.5, 0.5) {$F_{e_3}$};
        \node[above] at (0, 0.5) {$F_{e_4}$};
        \node[above] at (1.5, 0.5) {$F_{e_5}$};
        \node[above] at (3, 0.5) {$F_{1}$};
        \node[above] at (6, 0.5) {$F_{r - 5}$};
        \node[left] at (-1.5, -1.5) {$M$};

    \end{scope}

\end{tikzpicture}

%% file: main.bbl
\begin{thebibliography}{10}

\bibitem{Alon}
N.~Alon.
\newblock An extremal problem for sets with applications to graph theory.
\newblock {\em J. Combin. Theory Ser. A}, 40(1):82--89, 1985.

\bibitem{BFVW}
M.~D. Barrus, M.~Ferrara, J.~Vandenbussche, and P.~S. Wenger.
\newblock Colored saturation parameters for rainbow subgraphs.
\newblock {\em J. Graph Theory}, 86(4):375--386, 2017.

\bibitem{Bollobas}
B.~Bollob\'{a}s.
\newblock Weakly {$k$}-saturated graphs.
\newblock In {\em Beitr\"{a}ge zur {G}raphentheorie ({K}olloquium, {M}anebach,
  1967)}, pages 25--31. Teubner, Leipzig, 1968.

\bibitem{BJR}
N.~Bushaw, D.~Johnston, and P.~Rombach.
\newblock Rainbow {S}aturation.
\newblock {\em Graphs Combin.}, 38(5):Paper No. 166, 2022.

\bibitem{EHM}
P.~Erd{\H{o}}s, A.~Hajnal, and J.~W. Moon.
\newblock A problem in graph theory.
\newblock {\em Amer. Math. Monthly}, 71:1107--1110, 1964.

\bibitem{FFS}
J.~R. Faudree, R.~J. Faudree, and J.~R. Schmitt.
\newblock A survey of minimum saturated graphs.
\newblock {\em Electron. J. Combin.}, DS19(Dynamic Surveys):36, 2011.

\bibitem{FJLPSSSTT}
M.~Ferrara, D.~Johnston, S.~Loeb, F.~Pfender, A.~Schulte, H.~C. Smith,
  E.~Sullivan, M.~Tait, and C.~Tompkins.
\newblock On edge-colored saturation problems.
\newblock {\em J. Comb.}, 11(4):639--655, 2020.

\bibitem{GLP}
A.~Gir{\~{a}}o, D.~Lewis, and K.~Popielarz.
\newblock Rainbow saturation of graphs.
\newblock {\em J. Graph Theory}, 94(3):421--444, 2020.

\bibitem{HT}
D.~Hanson and B.~Toft.
\newblock Edge-colored saturated graphs.
\newblock {\em J. Graph Theory}, 11(2):191--196, 1987.

\bibitem{KT}
L.~K{\'{a}}szonyi and Z.~Tuza.
\newblock Saturated graphs with minimal number of edges.
\newblock {\em J. Graph Theory}, 10(2):203--210, 1986.

\bibitem{Korandi}
D.~Kor\'{a}ndi.
\newblock Rainbow saturation and graph capacities.
\newblock {\em SIAM J. Discrete Math.}, 32(2):1261--1264, 2018.

\bibitem{Lovasz}
L.~Lov\'{a}sz.
\newblock Flats in matroids and geometric graphs.
\newblock In {\em Combinatorial surveys ({P}roc. {S}ixth {B}ritish
  {C}ombinatorial {C}onf., {R}oyal {H}olloway {C}oll., {E}gham, 1977)}, pages
  45--86. Academic Press, London, 1977.

\bibitem{Pikhurko}
O.~Pikhurko.
\newblock Results and open problems on minimum saturated hypergraphs.
\newblock {\em Ars Combin.}, 72:111--127, 2004.

\bibitem{ST}
A.~Shapira and M.~Tyomkyn.
\newblock Weakly saturated hypergraphs and a conjecture of {T}uza.
\newblock {\em Proc. Amer. Math. Soc.}, 151(7):2795--2805, 2023.

\bibitem{TT}
M.~Truszczy\'{n}ski and Z.~Tuza.
\newblock Asymptotic results on saturated graphs.
\newblock {\em Discrete Math.}, 87(3):309--314, 1991.

\bibitem{Tuza}
Z.~Tuza.
\newblock A generalization of saturated graphs for finite languages.
\newblock In {\em Proceedings of the 4th {I}nternational {M}eeting of {Y}oung
  {C}omputer {S}cientists, {IMYCS} '86 ({S}molenice {C}astle, 1986)}, number
  185, pages 287--293, 1986.

\bibitem{TuzaExtremal}
Z.~Tuza.
\newblock Extremal problems on saturated graphs and hypergraphs.
\newblock {\em Ars Combin.}, 25(B):105--113, 1988.
\newblock Eleventh British Combinatorial Conference (London, 1987).

\bibitem{Zykov}
A.~A. Zykov.
\newblock On some properties of linear complexes.
\newblock {\em Mat. Sbornik N.S.}, 24(66):163--188, 1949.
\newblock [In Russian].

\end{thebibliography}
